\documentclass[11pt, reqno]{amsart}
  \usepackage{amsmath, amsfonts, amssymb,amsthm}
\numberwithin{equation}{section}
\newtheorem{theorem}{Theorem}[section]
\newtheorem{lemma}{Lemma}[section]
\newtheorem{cor}{Corollary}[section]

\theoremstyle{definition}

\newcommand{\C}{\mathbb{C}}

\begin{document}

\title{ Normality and uniqueness property of meromorphic function in terms of  some differential polynomials }

\author{Nguyen Viet Phuong}
\address{Thai Nguyen University of Economics and Business Administration, Vietnam} \email{nvphuongt@gmail.com and nguyenvietphuong@tueba.edu.vn}


\thanks{2010 Mathematics Subject Classification 30D35}

\begin{abstract}In this paper, we  will consider normality and uniqueness property of a family $\mathcal{F}$ of meromorphic functions when 
$[Q(f)]^{(k)}$ and $[Q(g)]^{(k)}$ share $\alpha$ ignoring multiplicities, for any $f,g\in \mathcal{F}$, where $Q$ is a polynomial and $\alpha $ is a small function. Our results do not need all of zeros of $Q$ have large order as  other authors's results. 
\end{abstract}

\thanks{This research is funded by Vietnam National Foundation for Science and Technology Development (NAFOSTED) under grant number  101.04-2017.320}

\baselineskip=16truept 
\maketitle 
\pagestyle{myheadings}
\markboth{}{}

\section{ Introduction and main results }
A family $\mathcal{F}$ of meromorphic functions defined in a plane domain $D$ is said to be {\it normal } on $D$, in the sense of Montel, if each sequence $\{f_n\}\subset\mathcal{F}$ contains a subsequence which converges spherically locally uniformly in $D.$ In 2004, Fang and Zalcman \cite{FZ} obtained an interesting result about the normality of a family of meromorphic functions and sharing values. In their paper, they consider  a family $\mathcal{F}$ of meromorphic functions such that their  zero's orders are at least $k +2$; and for each pair of functions $f$ and $g$ in $\mathcal{F}$ share $0$; and  $f^{(k)}$ and $g^{(k)}$ share a nonzero value $b$ in $D,$ then $\mathcal{F}$ is normal in $D.$  In 2008, Zhang \cite{QZ} improved the above result for the case of the first derivative by removing the condition that f and g share the value $0,$ and he considers for the case when $n\ge 3$ and $(f^{n})'$ and $(g^{n})'$ share a nonzero value $b$ in $D.$  Then, in 2009, Li and Gu \cite{LG} extended the results for higher derivative polynomials by considering for any positive integer $k$ and for $n\geq k+2$, if $(f^n)^{(k)}$ and $(g^n)^{(k)}$share a nonzero $b$ in $D$ for every pair of functions $f,g\in \mathcal{F},$ then $\mathcal{F}$ is normal in $D.$

 In this paper, we generalize the above results for differential polynomials as follows. Our result do not need the condition such that meromorphic functions have high zeros's orders.

\begin{theorem}\label{thm5} Let $\mathcal{F}$ be a family of meromorphic functions defined in a domain $D.$ Let $k, q\geq m\geq k+2$ be positive integers and let $Q(z)=z^q+a_{q-1}z^{q-1}+\dots+a_m z^m+a_0,$ where $a_0,a_m,\dots,a_{q-1}$ are complex constants and $a_m\ne 0.$ If for each pair of functions $f$ and $g$ in $\mathcal{F},$ $[Q(f)]^{(k)}$ and $[Q(g)]^{(k)}$ share a nonzero value $b$ in $D$ ignoring multiplicity, then $\mathcal{F}$ is normal in $D.$
\end{theorem}
 
As an immediately consequence of Theorem~\ref{thm5} when $q=m\geq k+2$ and $a_0=0,$ we obtain the following special case which recovers the known result in \cite{LG}.
\begin{cor}\cite[Theorem 1]{LG}\label{corollary2} 
 Let $\mathcal{F}$ be a family of meromorphic functions defined in a domain $D.$ Let $k, m\geq k+2$ be positive integers. If for each pair of functions $f$ and $g$ in $\mathcal{F},$ $(f^m)^{(k)}$ and $(g^m)^{(k)}$ share a nonzero value $b$ in $D,$ then $\mathcal{F}$ is normal in $D.$
\end{cor}

Next, we will give some conditions such that for meromorphic functions $f$ and $g$ in a family $\mathcal{F}$ if $[Q(f)]^{(k)}$ and $[Q(g)]^{(k)}$ share a small function
$\alpha,$  ignoring multiplicity, then $f$ must be either equal to $g$ or 
closely related to $g.$ Relating to this problem, in 2002, Fang and Fang \cite{FF} considered differential polynomials of  the form $f'f^n(f-1)^2,$ where $f$ is the meromorphic function. Recently, Li, Qiu and Xuan \cite{LQX} considered this problem to the case of higher order derivatives and expressions of the form $[f^nP(f)]^{(k)}.$  In 2017, An and Phuong \cite{AP} considered expressions of the form $[Q(f)]^{(k)},$ where
instead of assuming that the polynomials $Q$ have a high order zero, they 
assumed a more general hypothesis that $Q$ has a point of large ramification and small functions $\alpha$ without any restrictions on the zeros and poles.

 Denote by
$$Q'(z) = b\prod_{i=1}^{l}(z-\zeta_i)^{m_i}$$
 with $b\in\C^*,$ and denote by  $\upsilon$ and $h$  the indexes such that $1\le\upsilon\le  h\le l,$ and
\begin{align*}& m_1\ge m_2\ge \dots\ge m_\upsilon> k \ge m_{\upsilon+1}\ge \dots\ge m_l,\\
& m_1\ge m_2\ge \dots\ge m_h\ge k > m_{h+1}\ge \dots\ge m_l.\end{align*}

Considering the above, we obtain the following results.

\begin{theorem}\label{th2} 
 Let $f$ and $g$ be nonconstant meromorphic functions and $\alpha$ be a non-zero small function with respect to $f$.
Suppose $[Q(f)]^{(k)}$ and $[Q(g)]^{(k)}$ share $\alpha$  ignoring multiplicity. If 
$q>4k+12+\upsilon(5k+2)+5\sum_{i=\upsilon+1}^{l}m_i,$ and if one of the following holds 
\begin{enumerate}
\item[(i)]  $h\ge 4$; 
\item[(ii)] $h=3$ and $q\ne 2m_1-2k+2,$ \quad $q\ne \frac{3m_1-2k+3}{2},$ and $ q\ne 3m_i-2k+3,$ for all $i=1,2,3;$ or
\item[(iii)]  $h= 2$ and $f$ and $g$ share $\infty$ ignoring multiplicities,
\end{enumerate}
then 
$$Q(f)=Q(g)+c, \text{ for some constant } c.$$
\end{theorem}

Readers can find in \cite{AD, AWW, AZ, AZ1, F} conditions of $Q$ such that the equation   $Q( f ) = Q(g)+c$ has no nontrivial meromorphic function solution, i.e if   $Q(f)=Q(g),$ then $f=g.$

The polynomial $Q(z)$ is said to satisfy \textit{Hypothesis I} if
 $$Q(\zeta_i) \neq Q(\zeta_j)\ {\rm whenever}\ i\neq j, \  i, j = 1, 2, \ldots, l ,$$
 or in other words $Q$ is injective on the roots of $Q'$.

 As a consequence of Theorem~\ref{th2} and Theorem 1.4 in \cite{AP}, we obtain the following 
\begin{theorem}\label{th4}
 Let $f$ and $g$ be nonconstant meromorphic functions satisfying the conditions in Theorem~\ref{th2}. Let $Q(z)$ be a polynomial of degree at least 7.
\begin{enumerate}
\item[(i)] If there exists $i$ ($1\leq i\leq l$) such that $m_i>\frac{q+1}{2},$ then $c=0.$
\item[(ii)] If  $Q(z)$ satisfies the Hypothesis I then $f=g$, except when $Q(z)=(z-\zeta_1)^{m_1}(z-\zeta_2)$.
\end{enumerate}
\end{theorem}


In the special case that $Q(z)=z^nP(z),$ we recover  known result in \cite{LQX}, as special case of our result.

\begin{cor}\cite[Theorem 1]{LQX}\label{corollary1} 
 Let $f$ and $g$ be nonconstant meromorphic functions and $\alpha$ be a non-zero small function respect to $f$. Let $P$ be a polynomial of degree $m$.
Suppose $[f^nP(f)]^{(k)}$ and $[g^nP(g)]^{(k)}$ share $\alpha$ ignoring multiplicities. 
 If $n>4m+9k+14$
 then one of the following holds
\begin{enumerate}
\item[(i)] $f^nP(f)=g^nP(g)$; or
\item[(ii)] $[f^nP(f)]^{(k)}[f^nP(f)]^{(k)}=\alpha^2.$
\end{enumerate}
 \end{cor}

\section{Results Needed from  Nevanlinna's Theory}

We recall some standard definitions and results in Nevanlinna theory (see \cite{CY, R} for more detail). 

Let $f$ be a meromorphic function on $\C.$ Let $n(t,f)$ be the number of poles of $f(z)$ in $|z|\leq t,$ each counted with correct multiplicity,
and let $\overline{n}(t,f)$ denote the number of poles of $f(z)$
in $|z|\leq t,$ where each multiple pole is counted only once. The {\it counting function} of poles is defined as follows $$ N(r,f):= \int_{0}^{r}\Big [\frac{n(t,f)-n(0,f)}{t}\Big ]dt+n(0,f)\log r,$$ with a similar definition for $\overline{N}(r,f).$ The {\it proximity function} and {\it characteristic function} are defined respectively as follows: 
$$m(r,f):=\frac{1}{2\pi}\int_{0}^{2\pi}\log^{+}| f(re^{i\theta})|d\theta,  
\textnormal{~and}$$
 $$ T(r,f):=m(r,f)+N(r,f). $$
We denote by $N_{p)}(r,f)$ the counting function of poles of $f$ which have multiplicity at most $p$, each pole counted with its multiplicity, by $N_{(p}(r,f)$ the counting function of poles of $f$ which have multiplicity at least $p$, each pole counted with its multiplicity, and the corresponding reduced counting functions are denoted by $\overline{N}_{p)}(r,f)$ and $\overline{N}_{(p}(r,f).$ We denote by $\overline{N}_L(r,\dfrac{1}{f-\alpha})$ the counting function for zeros of both $f-\alpha$ and $g-\alpha$ about which $f-\alpha$ has larger multiplicity than $g-\alpha,$ with multiplicity not being counted, denote by $N_{E}^{1)}(r,\dfrac{1}{f-\alpha})$ the counting function for common simple zeros of both $f-\alpha$ and $g-\alpha,$ by $N_{E}^{1)}\big(r,\frac{1}{f-\alpha}|\alpha = 0\big)$ the counting function of common simple zeros of  both $f-\alpha$ and $g-\alpha$ which are zeros of $\alpha$ and by $N_{E}^{1)}\big(r,\frac{1}{f-\alpha}|\alpha = \infty\big)$ the counting function of common simple zeros of both $f-\alpha$ and $g-\alpha$  which are poles of $\alpha$ and  denote by $N_{E}^{1)}\big(r,\frac{1}{f-\alpha}|\alpha\ne 0,\alpha\ne\infty\big)$ the counting function of common simple zeros of both $f-\alpha$ and $g-\alpha$  which are not zeros and poles of $\alpha,$ and denote by $\overline{N}_{E}^{(2}(r,\dfrac{1}{f-\alpha})$ the counting function for common multiple zeros of both $f-\alpha$ and $g-\alpha$ with the same multiplicity, where multiplicity is not counted. Similarly, we have the notations $N_{p)}(r,g),$ $N_{(p}(r,g),$ $\overline{N}_{p)}(r,g),$ $\overline{N}_{(p}(r,g),$ $\overline{N}_L(r,\dfrac{1}{g-\alpha}),$ $N_{E}^{1)}(r,\dfrac{1}{g-\alpha}),$ $N_{E}^{1)}\big(r,\frac{1}{g-\alpha}|\alpha = 0\big),$ $N_{E}^{1)}\big(r,\frac{1}{g-\alpha}|\alpha = \infty\big),$ $N_{E}^{1)}\big(r,\frac{1}{g-\alpha}|\alpha\ne 0,\alpha\ne\infty\big)$ and $N_{E}^{(2}(r,\dfrac{1}{g-\alpha}).$

Let $a$ be a finite complex number, and let $p$ be a positive integer.
We denote by $N_{p}\Big (r,\frac{1}{f-a}\Big)$ the counting function for zeros of $f-a$ where
 a zero of multiplicity $m$ is counted $m$ times if $m \le p$
and $p$ times if $m > p$.

The logarithmic derivative lemma can be stated as follows (see \cite{R}).

\begin{lemma}[Logarithmic Derivative Lemma]  Let $f$ be a nonconstant meromorphic function on $\C.$ Then $$ m\left(r,\frac{f'}{f}\right)=S(r,f) $$ as $r\rightarrow\infty$ outside a subset of finite measure.
\end{lemma}

We state the first and second fundamental theorem in Nevanlinna theory (see e.g. \cite{H}, \cite{R}):

\begin{theorem}[First fundamental theorem] Let $f$ be a meromorphic function,
and let $c$ be a complex number. Then
$$ T\left(r,\frac{1}{f-c}\right)=T(r,f)+O(1). $$
\end{theorem}

\begin{theorem}[Second fundamental theorem]  Let $f$ be a nonconstant meromorphic function on $\C.$ Let $a_1,\cdots,a_q$ be  distinct meromorphic functions on $\C$. Assume that $a_i$ are small functions
with respect to $f$ for all $i=1, ...,q.$ Then, the inequality
$$ (q-2)T(r,f)\leq \sum_{j=1}^{q}\overline{N}\left(r,\frac{1}{f-a_{j}}\right)+S(r,f), $$ 
holds for all $r$ outside a set $E\subset (0,+\infty)$ 
with finite Lebesgue measure.\end{theorem}

\begin{lemma}\label{lemma3.1}\cite[Lemma 2.4]{ZY} Let $f$ be a nonconstant meromorphic function, and let $p$ and $k$ be two positive integers. Then
\begin{align*}
& N_p\Big (r, \frac{1}{f^{(k)}}\Big)\leq T(r,f^{(k)})-T(r,f)+N_{p+k}\Big (r, \frac{1}{f}\Big)+S(r,f),\\
& N_p\Big (r, \frac{1}{f^{(k)}}\Big)\leq k\overline{N}(r,f)+N_{p+k}\Big (r, \frac{1}{f}\Big)+S(r,f).
\end{align*}
Moreover, if $f^{(k)}\not\equiv 0,$ then
\begin{align*}
N\Big (r, \frac{1}{f^{(k)}}\Big)\leq k\overline{N}(r,f)+N\Big (r, \frac{1}{f}\Big)+S(r,f).
\end{align*}
\end{lemma}

\begin{lemma}\label{Lemma 3.3}\cite {Y} Let $f(z)$ be a nonconstant meromorphic function and let $a_n(\not\equiv 0),$ $ a_{n-1},\dots,a_0$ be small functions with respect to $f.$ Then
\begin{align*}
T(r,a_nf^n+a_{n-1}f^{n-1}+\cdots +a_0)=nT(r,f)+S(r,f).
\end{align*} \end{lemma}

\begin{lemma} \label{lemma3.6}\cite[Lemma 3.1]{AP} Let $Q$ be a polynomial of degree $q$ in $\C,$
and let $k$ be  positive integer. Let
$$Q'(z) = b\prod_{i=1}^{l}(z-\zeta_i)^{m_i}$$
 with $b\in\C^*.$ Let $f$ and $g$ be nonconstant meromorphic functions.
Assume that $[Q(f)]^{(k)}=[Q(g)]^{(k)}$. If $q-2l-2k-4>0$ then $Q(f)=Q(g) + c,$ for some constant $c$.
\end{lemma}

A meromorphic function $\alpha$ on $\C$ is called {\it a small function with respect to $f$} if it satisfies $T(r,\alpha)=S(r,f).$ We say that two meromorphic functions $f$ and $g$ {\it share a function $\alpha$ counting multiplicities} if $f-\alpha$ and $g-\alpha$ admit the same zeros with the same multiplicities, and we say that $f$ and $g$ {\it share $\alpha$ ignoring multiplicities} if we do not consider the multiplicities. 

\begin{lemma} \label{lemma3.4}\cite[Lemma 3.4]{AP}
Let $f,g$ be nonconstant meromorphic functions and  $\alpha (\not\equiv 0,\infty)$ be a small function with respect to $f$  and $g$. If $$[Q(f)]^{(k)}[Q(g)]^{(k)}=\alpha ^2,$$
then $h\le 2$ or $h=3$ and either $q=2m_1-2k+2,$ $q=\frac{3m_1-2k+3}{2},$ or $ q=3m_i-2k+3,\,$ for $i=1,2,$ and $3.$ 
If  we further assume that  $f$ and $g$ share $\infty$ ignoring
multiplicities,  then also $h = 1.$
\end{lemma}

\begin{lemma}\label{lemmaSmall}\cite[Lemma 3.2] {AP}
Let $f$ and $g$ be nonconstant meromorphic functions, and  let $\alpha$
be a small function with respect to $f$.
If  $q>\upsilon(k+1)+\sum_{i=\upsilon+1}^lm_i+6$ and $[Q(f)]^{(k)}$ and $[Q(g)]^{(k)}$ share $\alpha$ ignoring multiplicities, then $T(r,f)=O(T(r,g)),$ $T(r,g)=O(T(r,f)),$ and $\alpha$ is a small function with  respect to $g.$
\end{lemma}

\section{Proof of Theorem~\ref{thm5}}
Before giving a proof, we will recall some known results.

\begin{lemma}\label{Ligu}\cite[Lemma 2]{LG}
Let $k, m\geq k+2$ be positive integers, $b\ne 0$ be a finite complex number and $f$ be a nonconstant rational meromorphic function, then $(f^m)^{(k)}-b$ has at least two distinct zeros.
\end{lemma}

\begin{lemma}\label{transcendental}\cite[Theorem 10]{FW}
Let $f$ be a transcendental meromorphic function on $\C$ having at most
finitely many simple zeros, then $f ^{(k)}$ takes on every non-zero complex value infinitely often.
\end{lemma}

\begin{lemma}\label{constant}\cite[Lemma 10]{WF}
Let $f$ be a meromorphic function of finite order in $\C,$ $k$ be a positive integer. If the zeros of $f$ are of multiplicities at least $k+2$ and $f^{(k)}\ne 1,$ then $f$ is a constant.
\end{lemma}

\begin{lemma}\label{Zalcman}\cite{Za} 
Let $\mathcal{F}$ be a family of meromorphic functions in the unit disc $\Delta$ satisfying all zeros of functions in $\mathcal{F}$ have multiplicity greater than or equal to $\ell$ and all poles of functions in $\mathcal{F}$ have multiplicity greater than or equal to $j.$ Let $\alpha$ be a real number satisfying $-\ell<\alpha<j.$ Then if $\mathcal{F}$ is not normal at a point $z_0\in\Delta$ if and only if there exist
 \begin{enumerate}
\item[(i)] points $z_n\in\Delta,$ $z_n\rightarrow z_0;$
\item[(ii)] positive numbers $\rho_n,$ $\rho_n\rightarrow 0;$ and 
\item[(iii)] functions $f_n\in\mathcal{F}$
\end{enumerate}
such that $g_n(\zeta)=\rho_n^{\alpha}f_n(z_n+\rho_n\zeta)\rightarrow g(\zeta)$ spherically uniformly on compact subsets of $\C,$ where $g$ is a nonconstant meromorphic function and $g^{\#}(\zeta)\leq g^{\#}(0)=1.$  Moreover, the order of $g$ is not greater than $2.$ Here, as usual,$$g^{\#}(z)=\frac{|g'(z)|}{1+|g(z)|^2}$$ is the spherical derivative.
\end{lemma}

\begin{proof}[Proof of Theorem~\ref{thm5}] Suppose that $\mathcal{F}$ is not normal in $D.$ Then there exists at least one point $z_0\in D$ such that $\mathcal{F}$ is not normal at the point $z_0.$ Without loss of generality we assume that $z_0=0.$ By Lemma~\ref{Zalcman} for $\alpha =-\frac{k}{m},$ there exist points $z_n\rightarrow 0,$ positive numbers $\rho_n\rightarrow 0$ and functions $f_n\in\mathcal{F}$ such that 
\begin{equation}\label{5.1} g_n(\zeta)=\rho_n^{-\frac{k}{m}}f_n(z_n+\rho_n\zeta)\Longrightarrow g(\zeta)
\end{equation} 
locally uniformly with respect to the spherical metric, where $g$ is a nonconstant meromorphic function in $\C$ and its order is less than or equal to $2.$ 
From \eqref{5.1}, we obtain
\begin{equation*}
(g_n^{\nu}(\zeta))^{(k)}= [(\rho_n^{-\frac{k}{m}}f_n(z_n+\rho_n\zeta))^{\nu}]^{(k)}=\rho_n^{\frac{(m-\nu)k}{m}}[f_n^{\nu}(z_n+\rho_n\zeta)]^{(k)}, 
\end{equation*} for $\nu = m, \dots, q-1,q.$ Hence, we have
\begin{align}\label{5.2}
[Q(f_n)(z_n+\rho_n\zeta)]^{(k)}=&\rho_n^{\frac{(q-m)k}{m}}[g_n^q(\zeta)]^{(k)}+a_{q-1}\rho_n^{\frac{(q-m-1)k}{m}}[g_n^{q-1}(\zeta)]^{(k)}\notag\\
&+\dots+a_m[g_n^{m}(\zeta)]^{(k)}\Longrightarrow a_m[g^{m}(\zeta)]^{(k)}
\end{align} also locally uniformly with respect to the spherical metric.

If $a_m[g^{m}(\zeta)]^{(k)} -b \equiv 0$, then $g^m(\zeta)\equiv P_k(\zeta)$ where $P_k(\zeta)$ is a polynomial of degree $k,$ this contradicts the fact that $g$ is a nonconstant meromorphic function since $m\geq k+2.$ 
Thus, $a_m[g^{m}(\zeta)]^{(k)} -b \not\equiv 0.$

If $a_m[g^{m}(\zeta)]^{(k)} \ne b,$ then by Lemma~\ref{constant}, we obtain that $g$ is a constant which is a contradiction. Hence, $a_m[g^{m}(\zeta)]^{(k)} - b$ must have one zero at least. 

Now we claim that $a_m[g^{m}(\zeta)]^{(k)} - b$ has just a unique zero. In fact, suppose that there exist two distinct zeros $\zeta_0$ and $\zeta_0^*$ of $a_m[g^{m}(\zeta)]^{(k)} - b$ and choose $\delta>0$ small enough such that $D(\zeta_0,\delta)\cup D(\zeta^*_0,\delta)=\emptyset$ and $a_m[g^{m}(\zeta)]^{(k)} - b$ has no other zeros in $D(\zeta_0,\delta)\cap D(\zeta^*_0,\delta)$ except for $\zeta_0$ and $\zeta_0^*,$ where $D(\zeta_0,\delta)=\{\zeta\in\C:|\zeta-\zeta_0|<\delta\}$ and $D(\zeta^*_0,\delta)=\{\zeta\in\C:|\zeta-\zeta^*_0|<\delta\}.$

From \eqref{5.2}, by Hurwitz's theorem, there exist points $\zeta_n\in D(\zeta_0,\delta),$ $\zeta^*_n\in D(\zeta^*_0,\delta)$ such that for sufficiently large $n$, we have $$[Q(f_n)(z_n+\rho_n\zeta_n)]^{(k)}-b=0\quad\text{and}\quad [Q(f_n)(z_n+\rho_n\zeta^*_n)]^{(k)}-b=0.$$
By the assumption that $[Q(f_n)]^{(k)}$ and $[Q(f_1)]^{(k)}$ share $b$  in $D$ ignoring multiplicity for each $n,$ it follows that $$[Q(f_1)(z_n+\rho_n\zeta_n)]^{(k)}-b=0\quad\text{and}\quad [Q(f_1)(z_n+\rho_n\zeta^*_n)]^{(k)}-b=0.$$ By letting $n\rightarrow\infty,$ and noting $z_n+\rho_n\zeta_n\rightarrow 0$ and $z_n+\rho_n\zeta^*_n\rightarrow 0,$ we get $$[Q(f_1)(0)]^{(k)}-b=0.$$ Since the zeros of $[Q(f_1)]^{(k)}-b$ has no accumulation points, for sufficiently large $n,$ we have $$z_n+\rho_n\zeta_n= 0\quad\text{and}\quad z_n+\rho_n\zeta^*_n= 0.$$ Hence $$\zeta_n=-\frac{z_n}{\rho_n}\quad\text{and}\quad\zeta^*_n=-\frac{z_n}{\rho_n}.$$ This contradicts the fact that  $\zeta_n\in D(\zeta_0,\delta),$ $\zeta^*_n\in D(\zeta^*_0,\delta)$ and $D(\zeta_0,\delta)\cup D(\zeta^*_0,\delta)=\emptyset.$ Thus $a_m[g^{m}(\zeta)]^{(k)} - b$ has just a unique zero. This contradicts Lemma~\ref{transcendental} and Lemma~\ref{Ligu}. Theorem~\ref{thm5} is proved completely. 
\end{proof}

\section{Proofs of Theorem~\ref{th2}}

To prove the results, we need to prove the following lemmas.
\begin{lemma}\label{keylemma} Let $f$ and $g$ be two nonconstant meromorphic functions, and let $\alpha$ be a non-zero small function with respect to $f$ and $g.$ If $f$ and $g$ share $\alpha$  ignoring multiplicity. then one of the following three cases holds:
\begin{enumerate}
\item[(i)] $T(r,f)\leq N_2(r,f)+N_2(r,g)+N_2(r,\frac{1}{f})+N_2(r,\frac{1}{g})+2\big(\overline{N}(r,f)+\overline{N}(r,\frac{1}{f})\big )+\overline{N}(r,g)+\overline{N}(r,\frac{1}{g})+S(r,f)+S(r,g),$ and the same inequality holding for $T(r,g);$
\item[(ii)] $f\equiv g;$ or
\item[(iii)] $fg\equiv\alpha^2.$ 
\end{enumerate}
\end{lemma}

\begin{proof}
Set \begin{equation}F:=\frac{f}{\alpha},\, G:=\frac{g}{\alpha}\label{t1}\end{equation} and 
\begin{equation}H:=\frac{F''}{F'}-2\frac{F'}{F-1}-\frac{G''}{G'}+2\frac{G'}{G-1}\label{t2}
\end{equation}

Let $z_0\not\in\{z:\alpha(z)=0\}\cup\{z:\alpha(z)=\infty\}$ be a common simple zero of $f-\alpha$ and $g-\alpha$. Then, it follows from \eqref{t1} that $z_0$ is a common simple zero of $F-1$ and $G-1$. By a simple computation on local expansions shows that $H(z_0)=0.$ 

Suppose that $H\not\equiv 0.$ Since $f$ and $g$ share $\alpha$  ignoring multiplicity, we have
  \begin{align}\label{t3}
N^{1)}_{E}\big(r,\frac{1}{f-\alpha}\big)&=N^{1)}_{E}\big(r,\frac{1}{f-\alpha}|\alpha\ne 0,\alpha\ne\infty\big)+N^{1)}_{E}\big(r,\frac{1}{f-\alpha}|\alpha=0\big)\notag\\
&\quad+N^{1)}_{E}\big(r,\frac{1}{f-\alpha}|\alpha=\infty\big)\notag\\
&\le N^{1)}_{E}\big(r,\frac{1}{F-1}\big)+N\big(r,\frac{1}{\alpha}\big)+N(r,\alpha)\notag\\
&\le N\big(r,\frac{1}{H}\big)+S(r,f)\notag\\
&\le T(r,H) + S(r,f)\le N(r,H)+S(r,f)+S(r,g)
\end{align}

By \eqref{t2} we see that F has only simple poles. On the other hand, let $$z_1\not\in\{z:\alpha(z)=0\}\cup\{z:\alpha(z)=\infty\}$$ be a common multiple zero of $f-\alpha$ and $g-\alpha$ with the same multiplicity. Then, it follows from \eqref{t1} that $z_1$ is a common multiple zero of $F-1$ and $G-1$ with the same multiplicity and by calculating we get $H(z_1)\ne\infty$. In addition, by a  simple computation we can easily follow from \eqref{t2} that any simple pole of $F$ and $G$ is not a pole of $H.$ Therefore, by \eqref{t2}, the poles of $H$ only occur at zeros of $F'$ that are not the zeros of $F(F-1)$, zeros of $G'$ that are not the zeros of $G(G-1)$, the multiple zeros of $F$ and $G$ and the multiple poles of $F$ and $G,$ the common zeros of  $F-1$ and $G-1$ where their multiples are different and zeros or poles of $\alpha$. We denote by $N_{0}\big(r,\frac{1}{F'}\big)$ the counting function of zeros of $F'$ that are not the zeros of $F(F-1),$ and by $\overline{N}_{0}\big(r,\frac{1}{F'}\big)$ the corresponding reduced counting function. Similarly, we can define $N_{0}\big(r,\frac{1}{G'}\big)$ and $\overline{N}_{0}\big(r,\frac{1}{G'}\big).$ Hence, from the above observations, we obtain

\begin{align}\label{t4}
N(r,H)&\le \overline{N}_{(2}(r,F)+\overline{N}_{(2}(r,G)+\overline{N}_{(2}\big(r,\frac{1}{F}\big)+\overline{N}_{(2}\big(r,\frac{1}{G}\big)\notag\\
&\quad+\overline{N}_{L}\big(r,\frac{1}{F-1}\big)+\overline{N}_{L}\big(r,\frac{1}{G-1}\big)+\overline{N}_{0}\big(r,\frac{1}{F'}\big)+\overline{N}_{0}\big(r,\frac{1}{G'}\big)\notag\\
&\quad+\overline{N}(r,\alpha)+\overline{N}\big(r,\frac{1}{\alpha}\big)+S(r,f)+S(r,g)\notag\\
&\le \overline{N}_{(2}(r,F)+\overline{N}_{(2}(r,G)+\overline{N}_{(2}\big(r,\frac{1}{F}\big)+\overline{N}_{(2}\big(r,\frac{1}{G}\big)\notag\\
&\quad+\overline{N}_{L}\big(r,\frac{1}{F-1}\big)+\overline{N}_{L}\big(r,\frac{1}{G-1}\big)+\overline{N}_{0}\big(r,\frac{1}{F'}\big)+\overline{N}_{0}\big(r,\frac{1}{G'}\big)\notag\\
&\quad+S(r,f)+S(r,g).
\end{align}
Since $f$ and $g$ share $\alpha$  ignoring multiplicity, we have
\begin{align}\label{t5}
\overline{N}\big(r,\frac{1}{f-\alpha}\big)&=N^{1)}_{E}\big(r,\frac{1}{f-\alpha}\big)+\overline{N}^{(2}_{E}\big(r,\frac{1}{f-\alpha}\big)\notag\\
&\quad+\overline{N}_{L}\big(r,\frac{1}{f-\alpha}\big)+\overline{N}_{L}\big(r,\frac{1}{g-\alpha}\big)\notag\\
&\le N^{1)}_{E}\big(r,\frac{1}{f-\alpha}\big)+\overline{N}^{(2}_{E}\big(r,\frac{1}{F-1}\big)+\overline{N}_{L}\big(r,\frac{1}{F-1}\big)\notag\\
&\quad+\overline{N}_{L}\big(r,\frac{1}{G-1}\big)+S(r,f)+S(r,g)
\end{align}
where the inequality follows from the fact that
\begin{align*}\overline{N}^{(2}_{E}\big(r,\frac{1}{f-\alpha}\big)&\le\overline{N}^{(2}_{E}\big(r,\frac{1}{f-\alpha}|\alpha\ne 0,\infty\big)+\overline{N}^{(2}_{E}\big(r,\frac{1}{f-\alpha}|\alpha =0\big)\\
&\quad+\overline{N}^{(2}_{E}\big(r,\frac{1}{f-\alpha}|\alpha=\infty\big)\\
&\le\overline{N}^{(2}_{E}\big(r,\frac{1}{f-\alpha}|\alpha\ne 0,\infty\big)+N(r,\frac{1}{\alpha})+N(r,\alpha)\\
&\le \overline{N}^{(2}_{E}\big(r,\frac{1}{F-1}\big)+S(r,f),
\end{align*}
Similarly,
\begin{align*}\overline{N}_{L}\big(r,\frac{1}{f-\alpha}\big)&\le \overline{N}_{L}\big(r,\frac{1}{f-\alpha}\mid \alpha\ne 0,\infty\big)+\overline{N}_{L}\big(r,\frac{1}{f-\alpha}\mid \alpha= 0\big)\\
&\quad+\overline{N}_{L}\big(r,\frac{1}{f-\alpha}\mid \alpha = \infty\big)\\
&\le \overline{N}_{L}\big(r,\frac{1}{F-1}\big)+N\big(r,\frac{1}{\alpha}\big)+N(r,\alpha)\\
&\le\overline{N}_{L}\big(r,\frac{1}{F-1}\big)+S(r,f),
\end{align*} and $$\overline{N}_{L}\big(r,\frac{1}{g-\alpha}\big)\le \overline{N}_{L}\big(r,\frac{1}{G-1}\big)+S(r,g).$$
On the other hand, from the definition of $N_0(r,\frac{1}{G'})$ we follow that 
\begin{align}\label{t6}
N\big(r,\frac{1}{G'}\big )&\ge N_0\big(r,\frac{1}{G'}\big)+\overline{N}^{(2}_{E}\big(r,\frac{1}{F-1}\big)+\overline{N}_{L}\big(r,\frac{1}{G-1}\big)\notag\\
&\quad+N_{(2}\big(r,\frac{1}{G}\big)-\overline{N}_{(2}\big(r,\frac{1}{G}\big).
\end{align}
Combining \eqref{t6} and Lemma~\ref{lemma3.1}, we get
\begin{equation}\label{t7}
\overline{N}(r,G)+\overline{N}\big(r,\frac{1}{G}\big)+S(r,g)\ge \overline{N}^{(2}_{E}\big(r,\frac{1}{F-1}\big)+\overline{N}_{L}\big(r,\frac{1}{G-1}\big)+N_0\big(r,\frac{1}{G'}\big)
\end{equation}
Moreover, we have
\begin{align}\label{t8}
\overline{N}_L\big(r,\frac{1}{F-1}\big)&\le N\big(r,\frac{1}{F-1}\big)-\overline{N}\big(r,\frac{1}{F-1}\big)\notag\\
&\le \overline{N}\big(r,\frac{1}{F}\big)+\overline{N}(r,F)+S(r,f),
\end{align} where the second inequality follows from Lemma~\ref{lemma3.1} and the fact that
$$N\big(r,\frac{1}{F-1}\big)-\overline{N}\big(r,\frac{1}{F-1}\big)+N\big(r,\frac{1}{F}\big)-\overline{N}\big(r,\frac{1}{F}\big)\le N\big(r,\frac{1}{F'}\big).$$
By repeating the above argument, we can prove
\begin{align}\label{t9}\overline{N}_L\big(r,\frac{1}{G-1}\big)\le \overline{N}\big(r,\frac{1}{G}\big)+\overline{N}(r,G)+S(r,g),
\end{align}

On the other hand, we have
\begin{equation}\label{t11} \overline{N}\big(r,\frac{1}{F-1}\big)\le\overline{N}\big(r,\frac{1}{f-\alpha}\big)+\overline{N}(r,\alpha)\le \overline{N}\big(r,\frac{1}{f-\alpha}\big) +S(r,f).
\end{equation}
Applying the Second Main Theorem for $F$ and $0,\infty$ and $1$, we have 
\begin{align}\label{t12}
T(r,F)\le \overline{N}(r,F)+\overline{N}\big(r,\frac{1}{F}\big)+\overline{N}\big(r,\frac{1}{F-1}\big)-N_0\big(r,\frac{1}{F'}\big)+S(r,f).
\end{align}
Therefore, by combining \eqref{t4}, \eqref{t5} and \eqref{t7}-\eqref{t12} and using the fact that 
$$N_{2}(r,F)\le N_{2}(r,f)+N_{2}\big(r,\frac{1}{\alpha}\big)\le N_{2}(r,f)+S(r,f),$$
$$N_{2}\big(r,\frac{1}{F}\big)\le N_{2}\big(r,\frac{1}{f}\big)+N_{2}(r,\alpha)\le N_{2}\big(r,\frac{1}{f}\big)+S(r,f),$$
and similar inequalities 
\begin{align*}&N(r,F)\le N(r,f)+S(r,f),\quad  N\big(r,\frac{1}{F}\big)\le N\big(r,\frac{1}{f}\big)+S(r,f),\\
&N_{2}(r,G)\le N_{2}(r,g)+S(r,g),\quad N_{2}\big(r,\frac{1}{G}\big)\le N_{2}\big(r,\frac{1}{g}\big)+S(r,g),
\end{align*}
 we obtain
\begin{align}\label{t13}
T(r,F)&\le N_2(r,f)+N_2(r,g)+N_2\big(r,\frac{1}{f}\big)+N_2\big(r,\frac{1}{g}\big )+\overline{N}(r,g)+\overline{N}\big(r,\frac{1}{g}\big)\notag\\
&\quad+2\Big(\overline{N}(r,f)+\overline{N}\big(r,\frac{1}{f}\big)\Big)+S(r,f)+S(r,g).
\end{align}
On the other hand, 
\begin{equation}\label{t14}
T(r,F)=T\big(r,\frac{f}{\alpha}\big)\ge T(r,f)-T(r,\alpha)+O(1).
\end{equation}

Combining inequalities \eqref{t13} and \eqref{t14}, we obtain (i).

Suppose that $H\equiv 0$. We deduce from \eqref{t2} that
$$F=\frac{(b+1)G+a-b-1}{bG+a-b},$$ where $a,b$ are finite complex numbers and $a\ne 0.$ 

If $b\ne 0,-1$, then $$F-\frac{b+1}{b}=-\frac{a}{b(bG+a-b)}.$$ Applying the Second Main Theorem for $F$ and $0,\infty$ and $\frac{b+1}{b}$, we have
\begin{align*}T(r,F)&\le \overline{N}(r,F)+\overline{N}\big(r,\frac{1}{F-\frac{b+1}{b}}\big)+\overline{N}\big(r,\frac{1}{F}\big)+S(r,f)\\
&\le\overline{N}(r,F)+\overline{N}(r,G)+\overline{N}\big(r,\frac{1}{F}\big)+S(r,f)
\end{align*}
Hence, we get 
\begin{align*}T(r,f)&\le T(r,F)+S(r,f)\\
&\le \overline{N}(r,f)+\overline{N}(r,g)+\overline{N}\big(r,\frac{1}{f}\big)+S(r,f)+S(r,g),\end{align*} which implies (i).

If $b=0,$ then $F=\dfrac{G+a-1}{a}$. If $a=1,$ then $F=G$ which implies (ii). If $a\ne 1,$ applying the Second Main Theorem for $F$ and $0,\infty$ and $\frac{a-1}{a}$, we have
\begin{align*}T(r,f)&\le T(r,F)+S(r,f)\\
&\le \overline{N}(r,F)+\overline{N}\big(r,\frac{1}{F-\frac{a-1}{a}}\big)+\overline{N}\big(r,\frac{1}{F}\big)+S(r,f)\\
&\le\overline{N}(r,F)+\overline{N}(r,\frac{1}{G})+\overline{N}\big(r,\frac{1}{F}\big)+S(r,f)\\
&\le\overline{N}(r,f)+\overline{N}(r,\frac{1}{f})+\overline{N}\big(r,\frac{1}{f}\big)+S(r,f)+S(r,g).
\end{align*} We obtain (i).

If $b=-1,$ then $F=\dfrac{a}{a+1-G}.$ If $a\ne -1,$ applying the Second Main Theorem for $G$ and $0,\infty$ and $a+1,$ we have
\begin{align*}T(r,g)&\le T(r,G)+S(r,g)\\
&\le \overline{N}(r,G)+\overline{N}\big(r,\frac{1}{G-(a+1)}\big)+\overline{N}\big(r,\frac{1}{G}\big)+S(r,g)\\
&\le\overline{N}(r,G)+\overline{N}(r,F)+\overline{N}\big(r,\frac{1}{G}\big)+S(r,g)\\
&\le\overline{N}(r,f)+\overline{N}(r,g)+\overline{N}\big(r,\frac{1}{g}\big)+S(r,f)+S(r,g).
\end{align*} We get (i). If $a=-1,$ then $FG\equiv 1$ which implies (iii).
\end{proof}


\begin{theorem}\label{th1}
 Let $f$ and $g$ be nonconstant meromorphic functions, $\alpha$ be a non-zero small function with respect to $f.$ 
Suppose $[Q(f)]^{(k)}$ and $[Q(g)]^{(k)}$ share $\alpha$  ignoring multiplicity. If 
$q>4k+12+\upsilon (5k+2)+5\sum_{i=\upsilon+1}^{l}m_i,$
then one of the following holds:
\begin{enumerate}
\item[(i)] $Q(f)=Q(g)+c$, for some constant $c$. 
\item[(ii)] $[Q(f)]^{(k)}[Q(g)]^{(k)}=\alpha^2.$
\end{enumerate}
\end{theorem}

The conclusion (ii) in Theorem~\ref{th1} can be ruled out if we add more
constraints on  the multiple zeros of $Q'(z)$ or if  $f$ and $g$ share $\infty$ ignoring multiplicities. 

\begin{proof}[{Proof of Theorem~\ref{th1}}]
We denote
\begin{align*}
&F:=[Q(f)]^{(k)},\quad F_1:=Q(f),\\
&G:=[Q(g)]^{(k)},\quad  G_1:=Q(g).\end{align*}

It is easy  to see that 
$$S(r,{F})=S(r,f), \, \, \text{ and } \, \, S(r,G)=S(r,g).$$

By Lemma~\ref{lemmaSmall}, $\alpha$ is also a small function 
with respect to $g$.

Since $[Q(g)]^{(k)}$ share $\alpha$ ignoring multiplicities, by Lemma~\ref{keylemma}, one of the following cases holds:
\begin{enumerate}
\item[(i)] $T(r,F)\leq N_2(r,F)+N_2(r,G)+N_2(r,\frac{1}{F})+N_2(r,\frac{1}{G})+2\big(\overline{N}(r,F)+\overline{N}(r,\frac{1}{F})\big )+\overline{N}(r,G)+\overline{N}(r,\frac{1}{G})+S(r,f)+S(r,g),$ and the same inequality holding for $T(r,G);$
\item[(ii)] $F\equiv G;$ or
\item[(iii)] $FG\equiv\alpha^2.$ 
\end{enumerate}

If Case~(ii) holds then,  by Lemma~\ref{lemma3.6}, we have $Q(f)=Q(g)+c,$
for some constant $c$, which is the conclusion (i)
of the theorem.
If Case~(iii) holds, then we get the conclusion (ii) in the theorem. Therefore, we only have to consider when Case (i) holds,
which we now examine in more detail.

If Case~(i) holds, we have
\begin{align}\label{3.1}
T(r,F)&\le N_2(r,F)+N_2(r,G)+N_2(r,\frac{1}{F})+N_2(r,\frac{1}{G})+\overline{N}(r,G)+\overline{N}(r,\frac{1}{G}) \notag\\
&\quad+2\big(\overline{N}(r,F)+\overline{N}(r,\frac{1}{F})\big )+S(r,f)+S(r,g)\notag\\
&\leq N_2\Big (r,\frac{1}{{F}}\Big)
+ (k-1)\overline{N}(r,G_{1}')+N_{k+1}\Big ( r,\frac{1}{G_{1}'}\Big)
+N_2 (r,{F})+N_2 (r,{G})\notag\\
&\quad+2\Big((k-1)\overline{N}(r,F'_1)+N_k\big(r,\frac{1}{F'_1}\big)\Big)+2\overline{N}(r,F)\notag\\
&\quad +(k-1)\overline{N}(r,G'_1)+N_k\big(r,\frac{1}{G'_1}\big)+\overline{N}(r,G)+S(r,f)+S(r,g).
\end{align}
where the second  inequality follows from 
 Lemma~\ref{lemma3.1}, that
\begin{align*}
N_2\Big (r,\frac{1}{G}\Big)&=N_2\Big (r,\frac{1}{(G_{1}')^{(k-1)}}\Big)\cr
&\leq (k-1)\overline{N}(r,G_{1}')+N_{k+1}\Big ( r,\frac{1}{G_{1}'}\Big) +S(r,g),\cr
\overline{N}(r,\frac{1}{F}\Big)&=N_1\Big (r,\frac{1}{(F_{1}')^{(k-1)}}\Big)\cr
&\leq (k-1)\overline{N}(r,F_{1}')+N_{k}\Big ( r,\frac{1}{F_{1}'}\Big) +S(r,f),
\end{align*}
and $$\overline{N}(r,\frac{1}{G}\Big)\leq (k-1)\overline{N}(r,G_{1}')+N_{k}\Big ( r,\frac{1}{G_{1}'}\Big) +S(r,g).$$

On the other hand, we can write
\begin{equation*}
 Q(z)-R(z)=a(z-\beta)Q'(z) 
\end{equation*}
where $a\ne 0$ and $\beta$ are constants, and $R(z)$ is a polynomial of
degree at most $q-2.$ Applying the Logarithmic Derivative Lemma, we have 
\begin{align*} m\Big(r,\frac{1}{Q(f)-R(f)}\Big)&=m\Big(r,\frac{(Q(f))'}{Q(f)-R(f)}\cdot\frac{1}{(Q(f))'}\Big)\\
&\leq m\Big(r,\dfrac{f'}{a(f-\beta)}\Big)+m\Big(r,\frac{1}{F_{1}'}\Big)+O(1)\\
&\leq m\Big(r,\frac{1}{F_{1}'}\Big)+S(r,f),
\end{align*}
 which gives
\begin{align*} T(r,F_{1}')&= m(r,\frac1{F_1'})+N(r,\frac1{F_1'})+O(1)\\
&\geq T\Big(r,\frac{1}{Q(f)-R(f)}\Big)-N\Big(r,\frac{1}{Q(f)-R(f)}\Big)+N\Big(r,\frac{1}{F_{1}'}\Big)+O(1)\\
&\geq qT(r,f)- N\Big(r,\frac{1}{Q'(f)}\Big)-N\Big(r,\frac{1}{f-\beta}\Big)+N\Big(r,\frac{1}{F_{1}'}\Big)+O(1).
\end{align*}
Therefore,  applying Lemma~\ref{lemma3.1} to the function $F_1'$ 
(with the notation $(F_{1}')^{(k-1)}=F),$ we have
\begin{align}\label{3.2}
 T(r,F)&\geq T(r,F_{1}')+N_2\Big (r, \frac{1}{F}\Big)-N_{k+1}\Big (r, \frac{1}{F_{1}'}\Big)+S(r,f)\notag\\
   &\ge qT(r,f)- N\Big(r,\frac{1}{Q'(f)}\Big)-N\Big(r,\frac{1}{f-\beta}\Big)+N\Big(r,\frac{1}{F_{1}'}\Big)\notag\\
&\quad+N_2\Big (r, \frac{1}{F}\Big)-N_{k+1}\Big (r, \frac{1}{F_{1}'}\Big)+S(r,f).
\end{align}
The inequalities  \eqref{3.1} and \eqref{3.2}   imply 
\begin{align*}
qT(r,f)
 &\leq (k-1)\overline{N}(r,G_{1}')+N_{k+1}\Big ( r,\frac{1}{G_{1}'}\Big)+N_2 (r,G)+N_2 (r,F)\notag\\
&\quad+2\Big((k-1)\overline{N}(r,F'_1)+N_k\big(r,\frac{1}{F'_1}\big)\Big)+2\overline{N}(r,F)+(k-1)\overline{N}(r,G'_1)\notag\\
&\quad +N_k\big(r,\frac{1}{G'_1}\big)+\overline{N}(r,G)+N\Big(r,\frac{1}{f-\beta}\Big)+ N\Big(r,\frac{1}{Q'(f)}\Big)\notag\\
&\quad-N\Big(r,\frac{1}{F_{1}'}\Big)+N_{k+1}\Big (r, \frac{1}{F_{1}'}\Big)+S(r,f)+S(r,g)\cr
 &\le (2k+1)\overline{N}(r,g)+2N\Big ( r,\frac{1}{g'}\Big) +(2k+1)\sum_{i=1}^{\upsilon}N\Big ( r,\frac{1}{g-\zeta_i}\Big)\cr
&\quad +2\sum_{i=\upsilon+1}^{l}m_i N\Big ( r,\frac{1}{g-\zeta_i}\Big)+ (2k+2)\overline{N}(r,f)+2N\big(r,\frac{1}{f'}\big)\cr
&\quad +(3k+1)\sum_{i=1}^{\upsilon}N\Big ( r,\frac{1}{f-\zeta_i}\Big)+3\sum_{i=\upsilon+1}^{l}m_i N\Big ( r,\frac{1}{f-\zeta_i}\Big)\cr
&\quad +N\Big(r,\frac{1}{f-\beta}\Big)+S(r,f)+S(r,g)\cr
&\le \Big(2k+5+\upsilon (2k+1)+2\sum_{i=\upsilon+1}^{l}m_i\Big)T(r,g)\cr 
&\quad +\Big(2k+7+\upsilon (3k+1)+3\sum_{i=\upsilon+1}^{l}m_i\Big)T(r,f)\cr
&\quad +S(r,f)+S(r,g).
\end{align*}
Therefore
\begin{align}\label{ctp1}
&\quad\Big(q-2k-7-\upsilon (3k+1)-3\sum_{i=\upsilon+1}^{l}m_i\Big)T(r,f)\cr
&\le \Big(2k+5+\upsilon (2k+1)+2\sum_{i=\upsilon+1}^{l}m_i\Big)T(r,g)+S(r,f)+S(r,g).
\end{align}
By similar arguments, we have
\begin{align}\label{ctp2}
&\quad\Big(q-2k-7-\upsilon (3k+1)-3\sum_{i=\upsilon+1}^{l}m_i\Big)T(r,g)\cr
&\le \Big(2k+5+\upsilon (2k+1)+2\sum_{i=\upsilon+1}^{l}m_i\Big)T(r,f)+S(r,f)+S(r,g).
\end{align}
Combining \eqref{ctp1} and \eqref{ctp2}, we get
\begin{equation*}
 \Big(q-4k-12-\upsilon(5k+2)-5\sum_{i=\upsilon+1}^{l}m_i\Big)(T(r,g)+T(r,f))\le S(r,f)+S(r,g).
\end{equation*}
Thus, when $q>4k+12+\upsilon(5k+2)+5\sum_{i=\upsilon+1}^{l}m_i$ we have a contradiction.
\end{proof}

\begin{proof}[Proof of Theorem~\ref{th2}] The Theorem~\ref{th2} follows from Theorem~\ref{th1} and Lemma~\ref{lemma3.4}.
\end{proof}

\subsection{Proof of Corollary~\ref{corollary1}} 
We take $Q(z)=z^nP(z)$, and $q:=\deg Q=n+m$. We can write
$$Q'(z)=bz^{n-1}\prod_{j=2}^l(z-\zeta_j)^{m_j}$$ with $b\in\C^*.$ We will check that the hypotheses in the corollary imply
the hypotheses in the theorems.

We have $n>4m+9k+14=3m+q-n+9k+14$, hence $n>\dfrac{q+1}{2},$ and 
\begin{align*}q=n+m&>9k+5m+14=9k+5\sum_{j=2}^{\upsilon} m_j+5\sum_{j=\upsilon+1}^{l} m_j+14\cr
&\geq 9k+14+5(\upsilon-1)(k+1)+5\sum_{j=\upsilon+1}^{l} m_j\cr 
&=4k+9+3\upsilon+\upsilon(5k+2)+5\sum_{j=\upsilon+1}^{l} m_j\cr
&\geq 4k+12+\upsilon(5k+2)+5\sum_{j=\upsilon+1}^{l} m_j
\end{align*}
which satisfies the condition in Theorem~\ref{th1}. Hence, we are done for Corollary~\ref{corollary1}.

\end{document}